\documentclass[11pt]{amsart}

\usepackage[all,cmtip]{xy}
\usepackage{amsmath}
\usepackage{amssymb}
\usepackage{mathtools}
\usepackage[T1]{fontenc}
\usepackage[english]{babel}
\usepackage[hidelinks]{hyperref}
\usepackage[marginpar=2.5cm]{geometry}
\usepackage{xcolor}\usepackage{marginnote}
\usepackage{amsrefs}

\newtheorem{thm}{Theorem}[section]

\newtheorem*{thm*}{Theorem}
\newtheorem{lem}[thm]{Lemma}

\newtheorem{prop}[thm]{Proposition}
\newtheorem*{prop*}{Proposition}

\newtheorem{cor}[thm]{Corollary}

\theoremstyle{definition}
\newtheorem{defn}[thm]{Definition}

\newtheorem{remark}[thm]{Remark}

\newtheorem{example}[thm]{Example}

\def\bb{\mathbb}

\def\bb{\mathbb}

\DeclareMathOperator{\GL}{GL}

\DeclareMathOperator{\Span}{span}

\DeclareMathOperator{\Tr}{Tr}

\DeclareMathOperator{\rk}{rk}

\DeclareMathOperator{\conv}{conv}

\DeclareMathOperator{\OMAX}{OMAX}
\DeclareMathOperator{\OMIN}{OMIN}

\newcommand{\norm}[1]{\left\lVert #1 \right\rVert}

\newcommand\ip[2]{\left\langle #1\, , #2 \right\rangle}

\newcommand{\C}{\mathbb{C}}
\newcommand{\N}{\mathbb{N}}

\newcommand{\R}{\mathbb{R}}

\title[The Operator Systems of Hankel Matrices]{The Operator Systems of Hankel Matrices\\ and Complex Polynomials}

\author[S.~Qiu]{Shengwei Qiu}
\address{Department of Mathematics, Purdue University\\ 150 N. University Street, West Lafayette, IN 47907-2067}
\email{qiu221@purdue.edu}

\author[T.~Sinclair]{Thomas Sinclair}
\address{Department of Mathematics, Purdue University\\ 150 N. University Street, West Lafayette, IN 47907-2067}
\email{tsincla@purdue.edu}

\begin{document}

\begin{abstract}

Using polynomial matrix spectral factorization, we establish a complete order duality between $(n+1)\times(n+1)$ Hankel matrices and complex polynomials of degree at most $2n$ on the real line. The projection construction and trace pairing identify $e_n(t)=\sum_{j=0}^{n}t^{2j}$ as the polynomial order unit corresponding to the restriction of the matrix trace. The Cayley transform relates these matrix orders to Toeplitz matrices and truncated trigonometric polynomials, but for $n\geq 2$ the polynomial systems with units $e_n$ and $(1+t^2)^n$ are not unitally order isomorphic. We establish polynomial minimality and Hankel maximality, and characterize positivity of $C^*$-algebra-valued block Hankel matrices through completely positive maps on the polynomial operator system. An augmented Vandermonde decomposition describes polynomial states as trace-one positive Hankel matrices, with extreme states given by normalized evaluations and the leading-coefficient functional at infinity.

\end{abstract}

\maketitle

\section{Introduction}

Operator systems provide a fundamental framework for studying positivity in functional analysis and operator algebras. They capture the interplay between algebraic structure and order structure, allowing one to analyze completely positive maps, tensor products, and duality in a unified setting. In recent years, connections between operator systems and structured matrices have attracted increasing attention due to their applications in moment problems and harmonic analysis. The beautiful duality between operator systems of Toeplitz matrices and the spaces of truncated trigonometric polynomials as developed by Connes--van Suijlekom \cite{4} and Farenick \cite{5} is a motivating result in this direction for our study.

In this paper, we investigate finite Hankel matrices and complex polynomials on the real line through their matrix orders and distinguished order units. The Cayley transform identifies the matrix-ordered space of $(n+1)\times(n+1)$ Hankel matrices with the Toeplitz space and identifies the pointwise matrix order on polynomials of degree at most $2n$ with that on truncated trigonometric polynomials. It therefore relates the underlying cone duality to the results of Connes--van Suijlekom and Farenick. 

The choice of order unit, however, requires a separate analysis: the constant trigonometric unit corresponds under Cayley to $g_n(t)=(1+t^2)^n$, whereas the matrix trace pairing witnessing duality naturally selects $e_n(t)=\sum_{j=0}^n t^{2j}$.
For $n\geq2$, we show that the polynomial systems with units $e_n$ and $g_n$ are not unitally order isomorphic, although their matrix cones agree. Consequently, the polynomial system selected by the trace pairing is not unitally completely order isomorphic to the trigonometric polynomial system with constant unit. This specifies the aspect of the Hankel duality that is not obtained by simply transporting Farenick's unital isomorphism through the Cayley transform.

As in Farenick's work, polynomial matrix spectral factorization, a matrix version of the Fej\'er--Riesz theorem, is a key tool. It also gives a direct proof of the Hankel--polynomial duality. The polynomial system is minimal and the Hankel system is maximal. We apply the duality to positivity of block Hankel matrices with entries in a $C^*$-algebra. An augmented Vandermonde decomposition further identifies polynomial states with trace-one positive Hankel matrices, including the possible contribution of the leading-coefficient functional at infinity.

The paper is organized as follows. Section 2 reviews duality, minimal and maximal operator-system structures, and polynomial matrix spectral factorization. Section 3 introduces the Hankel and polynomial systems, establishes the Cayley correspondence and the duality theorem, and develops the projection construction. Section 4 identifies its polynomial realization and trace-induced order unit, and distinguishes the resulting unital isomorphism class from the one obtained through Cayley. Sections 5 and 6 treat block Hankel matrices and the Vandermonde description of the state space, respectively.

\section{Preliminaries}

In this paper, we assume that the reader is familiar with the concepts of matrix-ordered $*$-vector spaces, Archimedean matrix order units, unital complete order isomorphisms, minimal and maximal operator-system structures, $C^*$-envelopes, etc. See \cites{30,3,7,2,29}. We will give more precise references as the topics occur in the sequel.

\subsection{Duals of Operator Systems} 

Duality is an important structure that comes with an operator system and it plays a role in proving our main theorem. Given an operator system $S$, we write $S^d$ for the dual vector space. $S^d$ is equipped with an involution and a matrix order, which make it a matrix-ordered $*$-vector space. That is, 
\begin{enumerate}
\item For any $\phi \in S^d$, $s\in S$,
\[ 
    \phi: S \longrightarrow \C, \ \phi^*(s)= \overline{\phi(s^*)}.
\]
\item For $F =(f_{ij})\in M_n(S^d),$ define 
\[
    \phi_{F}: S \longrightarrow M_n(\C), \ \phi_{F}(s)= (f_{ij}(s)),
\]
$$\ C_n^d:=\{F: \phi_{F} \ \textup{is completely positive.}   \}.$$
\end{enumerate}

Equivalently, we may define the dual cones structure by
\[
C_n^d:=\{f:M_n(S) \longrightarrow \C : \textup{f is linear},\ f(C_n)\ge0  \}.
\]

For $S$ finite-dimensional, the equivalence of the above definitions gives the following canonical identifications:
$$M_n(S^d) \cong L(S^{dd}, M_n) \cong L(S, M_n) \cong (M_n(S))^d.$$

It is known that $S^d$ is always an operator system when $S$ is finite-dimensional, because the existence of an Archimedean matrix order unit is guaranteed \cite{7}. However, the order unit of $S^d$ is not unique. If a positive linear functional $\phi: W \longrightarrow \C$, where $W$ is an ordered $\ast$-vector space with an Archimedean order unit, satisfies $\phi(e)=1$, then we call it a \emph{state}. If it has a trivial kernel on positive elements, i.e., $\forall x\in W^+\setminus\{0\}$, $\phi(x)>0$,  then we call it a \emph{faithful state}. We denote $\mathbb{S}(W)$ as the space of states of W. By Proposition 2.17 in \cite{4}, for any finite-dimensional operator system $S$, any faithful state on $S$ can serve as an Archimedean order unit for $S^d$.

By the preceding observation, we get an efficient way to characterize an Archimedean order unit of the dual operator system. The following two theorems are from \cite{7}:

\begin{thm}\label{4}

Given any finite-dimensional operator system S, $(S^{d})^{d}$ is completely order isomorphic to $S$, i.e., $(S^{d})^{d} \cong S$.

\end{thm}

\begin{thm}\label{5}

Given two operator systems $S_1$ and $S_2$, any linear map $\phi: S_1 \longrightarrow S_2$ is completely positive if and only if its adjoint map $\phi^{d}: S_2^{d} \longrightarrow S_1^{d}$ is completely positive.

\end{thm}

\subsection{Minimal and Maximal Operator Systems} 

For any ordered $\ast$-vector space $S$ with an Archimedean order unit, there are numerous ways to construct its matrix order above its ``ground level''. There is a resulting partial order on the matrix order structures given by reverse level-wise inclusion of the positive cones. To be precise, for two sequences of cones of an ordered $*$-vector space with an Archimedean order unit $C=\{C_n\}$ and $D=\{D_n\}$, if $C_n \subseteq D_n$, for all $n$, then we say that $C$ is \emph{stronger} than $D$, or $D$ is \emph{weaker} than $C$. 

Among these operator systems, there are two important operator system structures, the maximal and the minimal, denoted $\OMAX(S)$ and $\OMIN(S)$ respectively. These were first defined and investigated by Paulsen, Todorov, and Tomforde \cite{11}. The maximal operator system is the strongest order structure of an ordered $\ast$-vector space with an Archimedean order unit that can be equipped with, that is, it contains the level-wise smallest positive cones among all operator system structures which agree with the order structure on $S$. The minimal operator system is the weakest order structure of an ordered $*$-vector space with an Archimedean order unit, possessing the level-wise largest positive cones under the same constraint of agreeing with the order structure on $S$. 

The formal definition of minimal cones is as follows:

\begin{defn}\label{7}
Given an ordered $*$-vector space $(S, C, e)$ with positive cone $C$ and Archimedean order unit $e$, we define $C_k^{min}\subset M_k(S)$ by
\[
C_k^{min}:= \left\{(a_{ij}):\sum_{i,j=1}^{k}\overline{x_i}x_ja_{ij} \in C, \quad\forall x_i \in \C\right\}.
\]

The operator system $(S, \{C_k^{min}\}_{k\ge 1}, e)$ is called the \emph{minimal operator system structure}, $\OMIN(S)$, on $S$.
    
\end{defn}

The minimal operator system structure may be characterized as follows. 
Noting that $\mathbb{S}(V)$ is a compact Hausdorff space with respect to the weak-$*$ topology, the map 
\[
\phi : V \longrightarrow C(\mathbb{S}(V)), \qquad \phi (v)(f) := f(v)
\]
is a unital order isomorphism onto its range, where $C(\mathbb S(V))$ is given its usual operator system structure as a unital C$^*$-algebra. Then, $\phi$ is a complete isometry exactly with respect to the minimal order norm on $V$. See \cite{11}, Theorem 5.2, for a proof.

\medskip

The formal definition of the maximal operator system structure is as follows:

\begin{defn}\label{9}

Given an ordered $*$-vector space $(S,C,e)$ , denote the cone
\[
D_k^{max} := \left\{\sum_{i=1}^{m}a_i \otimes b_i: a_i\in M_k(\bb C)^+,\ b_i\in C,\ m\in \N \right\}.
\]

Then the maximal cone is defined as
\[
C_k^{max} := \left\{ X\in M_k(S)_{sa}: X + r(I_k\otimes e) \in D_k^{max},\, \forall r>0\right\}.
\]

The operator system $(S, \{C_k^{max}\}_{k\ge 1}, e)$ is called the \emph{maximal operator system}, denoted by $\OMAX(S)$.

\end{defn}

\subsection{Polynomial Matrix Spectral Factorization}

One version of the Fej\'er--Riesz theorem, called polynomial matrix spectral factorization, plays an important role in proving the main result of the paper. For other versions, we refer the reader to~\cite{8,9}. The polynomial matrix spectral factorization theorem provides a way to decompose a positive matrix on the real line whose entries are polynomials of complex coefficients. 

By \cite{1}, Theorem 2, we have:

\begin{thm}[Matrix Fej\'er--Riesz]
\label{14}
Let $P$ be an $m\times m$ matrix of polynomials in a real variable $t$ with degree at most $2n$, i.e., $P(t)=\sum_{k=0}^{2n}p_kt^k$, where $p_k \in M_m(\mathbb C)$. If $P(t)$ is positive semidefinite for all $t\in\mathbb R$, then there exists a $m\times m$ polynomial matrix $Q(t)=\sum_{k=0}^{n}q_kt^k$, such that $P(t)=Q(t)^*Q(t)$.

\end{thm}

\section{The Operator Systems of Finite Hankel Matrices and Polynomials}

In this section, we introduce the Hankel and polynomial operator systems, compare their matrix orders with the Toeplitz and trigonometric polynomial systems through the Cayley transform, and prove the Hankel--polynomial duality. We then express this duality through projected matrix cones and the trace pairing. Throughout, we distinguish the choice of a strictly positive Hankel unit from the polynomial unit induced by the matrix trace.

\subsection{Operator Systems of Hankel Matrices} 

We denote the vector space of $(n+1)\times (n+1)$ Hankel matrices with complex entries as $\mathcal H_{n}$, that is, all matrices of the form $H = (h_{i+j})_{i,j=0}^{n}$. In schematic matrix form this is:
\[
H=
\begin{bmatrix}
h_0 & h_1 & \cdots & h_{n} \\
h_1&  h_2 & \cdots & h_{n+1} \\
\vdots & \vdots & \ddots & \vdots \\
h_{n} & h_{n+1} & \cdots & h_{2n}
\end{bmatrix}.
\]

Thus $\dim(\mathcal H_n)=2n+1$. We have that $\mathcal H_n$ inherits the usual matrix-ordered $\ast$-vector space structure from $M_{n+1}(\mathbb C)$. Note that the standard order unit, $I_{n+1}$, for $M_{n+1}(\mathbb C)$ does not belong to $\mathcal H_n$ for $n\geq 2$.
However, as the next lemma shows, any strictly positive Hankel matrix $\eta$ can serve as an Archimedean matrix order unit. Example~\ref{11} supplies one such choice. The matrix cones and coefficient pairing are independent of this choice, while the unit on the dual polynomial space is $L_\eta$. Separately, the faithful functional $\Tr|_{\mathcal H_n}$ determines the polynomial unit $e_n$ under the duality, as developed in Section~\ref{app:trace-unit}.

\begin{lem} \label{19}

Any strictly positive Hankel matrix is an Archimedean matrix order unit of the matrix-ordered $*$-vector space $\mathcal H_n$.

\end{lem}

\begin{proof}
If $\eta>0$, then there exists $c>0$ such that $$\eta\geq cI.$$
Hence, for every self-adjoint $X\in M_m(\mathcal H_n)$,
$$I_m\otimes \eta\geq cI,$$
so $$r(I_m\otimes \eta)+X\ge0$$
for r sufficiently large. This proves $I_m\otimes \eta$ is an order unit.

For the Archimedean property, if $$X+\epsilon(I_m\otimes \eta) \ge0 \ \ \ for \ every\  \epsilon>0,$$
then closedness of the positive cone gives $X\ge0$ as $\epsilon \downarrow 0.$

\end{proof}

\begin{example}\label{11}

Define the entries of a $(n+1)\times (n+1)$ Hankel matrix $H = (h_{i+j})_{i,j=0}^{n}$, 
$$h_k=\frac{1}{\sqrt{2\pi}}\int_{-\infty}^{+\infty}t^k\, e^{\frac{-t^2}{2}}dt$$ as the moment of standard normal distribution, i.e.
\[
h_k = 
\begin{cases}
(k-1)!!, & \text{if } k \hspace{0.15 cm} is \hspace{0.15 cm}even \\
0, & \text{if } k\hspace{0.15 cm}is \hspace{0.15 cm}odd
\end{cases}
\]
with $1!!=1$ by convention

Then, for every nonzero vector $x=(x_0,x_1, \dots, x_{n}) \in \mathbb{C}^{n+1}$, 
\begin{align*}
    \langle Hx,x \rangle= \sum_{i,j}h_{i+j}x_j \overline{x_{i}} &=\sum_{i,j}x_j \overline{x_{i}}\frac{1}{\sqrt{2\pi}}\int_{-\infty}^{+\infty}t^{i+j} e^{\frac{-t^2}{2}}dt \\
    &=\frac{1}{\sqrt{2\pi}}\int_{-\infty}^{+\infty}\left|\sum_{i=0}^nx_it^i\right|^2 e^{\frac{-t^2}{2}}dt >0.
\end{align*}
This proves directly that H is positive definite.

\end{example}

\subsection{Operator System of Polynomials}

We denote the $\ast$-vector space of polynomials of complex coefficients of a real variable with degree at most $2n$ by $\mathcal P_{2n}$, equipped with the involution $p^*(t)=\sum_{k=0}^{2n} \overline{a_k}t^k$, where $p(t)=\sum_{k=0}^{2n} a_kt^k$, $t \in \mathbb{R}$ and each $a_k \in \mathbb{C}$. The matrix-order structure is defined by the positive cones ${\{C_k'\}}$, where
\[
    C_k':= \left\{P(t)\in M_k(\mathcal P_{2n}): \langle P(t)x,x\rangle \ge 0,\,  \forall t \in \mathbb{R},\, \forall x \in \mathbb{ C}^{k}\right\}.
\]
We choose
\[e_n(t) :=\sum_{j=0}^{n}t^{2j}.\]
For the column $v(t)=(1,t,\ldots,t^n)^T$, set $q(A)(t)=v(t)^*Av(t)$. Then $q(I_{n+1})=e_n$. Section~\ref{app:trace-unit} shows that $q$ factors through the trace-orthogonal projection onto $\mathcal H_n$ and realizes the polynomial matrix cones as projected matrix cones. In particular, $e_n$ corresponds to $\Tr|_{\mathcal H_n}$ under the duality. We first verify directly that it is an Archimedean matrix order unit. Let $$p(t)= \sum_{k=0}^{2n} a_k t^k \in M_k(\mathcal P_{2n})_{sa},\ \ \ C=\sum_{k=0}^{2n}\norm{a_k}.$$ 
Since $|t|^k\leq e_n(t)$ for all $0\leq k\leq 2n$, 
we have $$\norm{p(t)}\leq Ce_n(t)$$
for every $t\in \mathbb R$. Therefore
$$-Ce_n(t)I_k\leq p(t) \leq Ce_n(t)I_k.$$
This proves the matrix order-unit property immediately. 
Now let us verify that it is an Archimedean matrix order unit. Denote $$(e_n)_k:= I_k\otimes e_n.
$$

We claim that $(e_n)_k$ is an Archimedean order unit for $M_k(\mathcal P_{2n})$, for all $k\in \N$. Indeed, if $P + r e_k\in C_k'$ for all $r>0$, then $\ip{P(t)x}{x}\ge -r\|x\|^2$ for for all $r> 0$, hence $\ip{P(t)x}{x}\ge 0$ for all $t\in\bb R$ and $P\in C_k'$. Thus, we have  

\begin{prop}\label{prop:P2n-operator-system}
    For each $n\in\N$, $(\mathcal P_{2n},\{C'_k\}_{k\ge 1},e_n)$ is an operator system.
\end{prop}

\begin{prop}\label{OMIN}

The operator system of polynomials defined above is the minimal operator system structure on $\mathcal P_{2n}$.

\end{prop}

\begin{proof}

For any $Q_m(t)=(q_{ij}(t)) \in C'_m$, $Q_{m}(t)\ge0$, for all $t \in \R$ $\Leftrightarrow \sum_{i,j=1}^{m}\overline {x_i}x_j q_{ij}(t)\ge0, \forall x_1,\dots, x_m \in \C, \forall t\in \R$, which implies that $Q_m(t) \in C_m^{min}$. 

Conversely, suppose that $Q_m(t) \in C_m^{min}$, $\sum_{i,j=1}^{m}\overline {x_i}x_j q_{ij}(t)\ge 0$, $\forall x_1,\dots, x_m \in \C$, $\forall t \in \R$. For all $t \in \R$, $Q_m(t)\ge0$, so $Q_m(t)\in C'_m$. \qedhere

\end{proof}

\begin{remark} Verifying its minimality as an abstract operator system $(\mathcal P_{2n}, C_k',e)$, we can embed it into $C(\mathbb{S}(\mathcal P_{2n}))$ as a concrete operator system by using the Kadison representation theorem. The minimal structure of the operator system of polynomials will eventually lead to the commutativity of its $C^*$-envelope, which will be used in the sequel.
\end{remark}

\subsection{The Hankel--Toeplitz correspondence}
\label{app:cayley}
We next relate the two matrix orders to the Toeplitz and trigonometric polynomial systems of \cites{4,5}. The Cayley transform gives these identifications explicitly, but the constant trigonometric unit corresponds to $(1+t^2)^n$, which differs from the trace-induced unit $e_n$ when $n\geq2$. Keeping track of the units will make clear how the duality below relates to the Toeplitz result.

Throughout this subsection, $n\geq1$,
$v(t)=(1,t,\ldots,t^n)^T$ is a column vector, and
$L_H(\sum a_kt^k)=\sum a_kh_k$ for $H=(h_{i+j})$.
Write
\[
\mathcal T_{n+1}=\{(c_{i-j})_{i,j=0}^n:c_{-n},\ldots,c_n\in\mathbb C\},
\qquad
\mathcal F_n=\Span\{z^k:-n\leq k\leq n\}\subset C(\mathbb T).
\]
The Toeplitz space has its inherited matrix cones and unit $I_{n+1}$,
and $\mathcal F_n$ has its pointwise matrix cones and unit $1$.
These are the Toeplitz and trigonometric polynomial systems appearing
in \cites{4,5}, with the indexing specified here explicitly.
Set $g_n(t)=(1+t^2)^n$.

\begin{prop}\label{app:prop-cayley}
Let $B\in M_{n+1}(\mathbb C)$ be the coefficient matrix determined by
\[
(Bv(t))_j=(t+i)^{n-j}(t-i)^j,\qquad 0\leq j\leq n.
\]
Then $B$ is invertible, and
\[
\Theta:\mathcal H_n\longrightarrow\mathcal T_{n+1},
\qquad \Theta(H)=BHB^*,
\]
is a linear $*$-isomorphism of matrix-ordered vector spaces.
Furthermore, the formula
\begin{equation}\label{app:eq-polynomial-cayley}
\Psi(p)(z)=\frac{p(t)}{g_n(t)},\qquad
z=\frac{t-i}{t+i},\qquad
\Psi(p)(1)=[t^{2n}]p,
\end{equation}
defines a unital complete order isomorphism
$(\mathcal P_{2n},\{C'_k\},g_n)\cong(\mathcal F_n,1)$.
\end{prop}

\begin{proof}
Put $u(z)=(1,z,\ldots,z^n)^T$. For real $t$ we have
\[
Bv(t)=(t+i)^nu(z),\qquad
Bv(t)v(t)^*B^*=g_n(t)u(z)u(z)^*.
\]
The polynomials defining the rows of $B$ are linearly independent:
dividing a putative linear relation by $(t+i)^n$ gives a polynomial
in $z$ vanishing at infinitely many points of $\mathbb T$. Thus $B$
is invertible. The matrices $v(t)v(t)^*$, for real $t$, span
$\mathcal H_n$, since their anti-diagonal coordinates are
$(1,t,\ldots,t^{2n})$. Each $u(z)u(z)^*$ is Toeplitz, so congruence
maps $\mathcal H_n$ into $\mathcal T_{n+1}$. Both spaces have complex
dimension $2n+1$, proving surjectivity. At matrix level $k$ the map
is congruence by $I_k\otimes B$, proving both positivity assertions.

To identify the range of $\Psi$, every $p\in\mathcal P_{2n}$ can be
written $p(t)=v(t)^*Av(t)$ for some $A\in M_{n+1}(\mathbb C)$.
Consequently
\[
\frac{p(t)}{g_n(t)}
=u(z)^*(B^{-1})^*AB^{-1}u(z)\in\mathcal F_n.
\]
This defines an injective map between spaces of the same dimension,
hence a bijection. Formula \eqref{app:eq-polynomial-cayley} at $z=1$
is the limit as $|t|\to\infty$. Division by the strictly positive
scalar $g_n(t)$ preserves pointwise matrix positivity, and positivity
at $z=1$ follows by continuity. The inverse also preserves positivity
at every matrix level, and $\Psi(g_n)=1$.
\end{proof}

\begin{remark}[The order units under Cayley]\label{app:rmk-cayley-units}
There is a particular Hankel unit for which $\Theta$ is unital:
\begin{equation}\label{app:eq-cayley-unit}
\eta_C=B^{-1}(B^{-1})^*,\qquad
(\eta_C)_{ij}=\frac1\pi\int_{\mathbb R}
\frac{t^{i+j}}{(1+t^2)^{n+1}}\,dt.
\end{equation}
Indeed, the integral defines a strictly positive Hankel matrix.
Applying $\Theta$ and using the change of variables
$z=(t-i)/(t+i)$ gives
\[
\Theta(\eta_C)=\int_{\mathbb T}u(z)u(z)^*\,dm(z)=I_{n+1},
\]
where $m$ is normalized Haar measure. This proves
\eqref{app:eq-cayley-unit}, including its normalization.
For an arbitrary strictly positive Hankel unit $\eta$, $\Theta$ is
unital onto $(\mathcal T_{n+1},B\eta B^*)$.

On the polynomial side the unit in this paper becomes
\[
\Psi(e_n)(z)=\frac{\sum_{j=0}^nt^{2j}}{(1+t^2)^n}.
\]
It is a strictly positive element of $\mathcal F_n$, and is
nonconstant when $n\geq2$: its value at $t=0$ is $1$, whereas its
value at $t=1$ is $(n+1)/2^n<1$. Thus, for $n\geq2$, the displayed Cayley map does not identify the trace-induced polynomial unit with $1$. Section~\ref{app:trace-unit} explains this distinction through the projection construction and the trace pairing; Proposition~\ref{app:prop-distinct-units} shows that no other unital order isomorphism identifies these two polynomial systems either. We now establish the Hankel--polynomial duality directly, with the units specified by this pairing.
\end{remark}

\subsection{The Isomorphism Theorem}


In the following theorem, $\mathcal H_n$ has its inherited matrix cones and a chosen strictly positive Hankel order unit $\eta$, while $\mathcal P_{2n}$ has its pointwise matrix cones and order unit $e_n=\sum_{j=0}^{n}t^{2j}$. We equip $\mathcal P_{2n}^d$ with the faithful positive functional $L_\eta$ as order unit, and $\mathcal H_n^d$ with $\Tr|_{\mathcal H_n}$. These conventions specify both unital identifications in the theorem.

\begin{thm}\label{Iso} 

With these choices of order unit, the coefficient pairing $H\mapsto L_H$ is a unital complete order isomorphism $\mathcal H_n\cong\mathcal P_{2n}^d$. Its dual identifies $\mathcal P_{2n}$ unitally and completely order isomorphically with $\mathcal H_n^d$, sending $e_n$ to $\Tr|_{\mathcal H_n}$.

\end{thm}

\begin{proof}
    
Consider a polynomial $p(t)=\sum_{k=0}^{2n} a_kt^k$. The linear map
\[
\phi: \mathcal H_n \longrightarrow (\mathcal P_{2n})^*,
\]
maps any Hankel matrix $H=(h_{i+j})_{i,j=0}^n$ to the dual of polynomials by
\[
\phi(H)=L_{H}, \qquad L_{H}(p)=\sum_{k=0}^{2n} a_kh_{k}.
\]

Denote the basis of $\mathcal H_n$ by $\{b_0, b_1, \cdots, b_{2n}\}$, where ${b_i}$ has entry $1$ on the $i$-th anti-diagonal and $0$ on the other entries. We observe that $\{L_{b_0}, \dotsc, L_{b_{2n}}\}$ is the dual basis to $\{1,\dotsc, t^{2n}\}$, that is,
\[
L_{b_i}(t^j) = 
\begin{cases}
1 & \text{if } i = j, \\
0 & \text{if } i \ne j.
\end{cases}
\]

We first show that $\phi$ is completely positive. Let
\[
H=\sum_{k=0}^{2n} N_k\otimes b_k\in M_m(\mathcal H_n)
\]
be positive. Equivalently, the block Hankel matrix
\[
(N_{i+j})_{i,j=0}^n\in M_{n+1}(M_m)
\]
is positive. We must show that
\[
\phi^{(m)}(H)=\sum_{k=0}^{2n}N_k\otimes L_{b_k}
\]
is a positive element of $M_m((\mathcal P_{2n})^d)$. By the definition of the matrix order on the dual, this means that the associated map
\[
\psi_H:\mathcal P_{2n}\longrightarrow M_m,\qquad \psi_H(t^k)=N_k,
\]
is completely positive.

Let $r\in \mathbb N$, and let
\[
P(t)=\sum_{k=0}^{2n}\,p_kt^k\in M_r(\mathcal P_{2n})
\]
be positive, i.e. $P(t)\geq 0$ for all $t\in\mathbb R$. By the polynomial matrix spectral factorization theorem, there exists a matrix polynomial
\[
Q(t)=\sum_{i=0}^{n}\,q_it^i,\qquad q_i\in M_r(\mathbb C),
\]
such that $P(t)=Q(t)^*Q(t)$. Therefore
\[
p_k=\sum_{i+j=k}q_i^*q_j.
\]
Consequently,
\[
\psi_H^{(r)}(P)=\sum_{k=0}^{2n}\,p_k\otimes N_k
 =\sum_{i,j=0}^{n}\,q_i^*q_j\otimes N_{i+j}.
\]
We claim that this last element is positive in $M_r\otimes M_m$. Indeed, if $\xi\in\mathbb C^r\otimes\mathbb C^m$ and we set
\[
\eta_i := (q_i\otimes I_m)\xi,
\]
then
\begin{align*}
\left\langle \left(\sum_{i,j=0}^{n}q_i^*q_j\otimes N_{i+j}\right)\xi,\xi\right\rangle
&=\sum_{i,j=0}^{n}\left\langle (I_r\otimes N_{i+j})\eta_j,\eta_i\right\rangle\geq 0,
\end{align*}
because $(N_{i+j})_{i,j=0}^n\geq0$. Thus $\psi_H$ is completely positive, and so $\phi$ is completely positive.

We now show that $\phi^{-1}$ is completely positive. Let
\[
F=\sum_{k=0}^{2n}\, N_k\otimes L_{b_k}\in M_m((\mathcal P_{2n})^d)
\]
be positive. This means precisely that the associated map $\psi_F$
is completely positive. Consider the matrix polynomial
\[
M(t) :=\big(t^{i+j}\big)_{i,j=0}^{n}=v(t)^*v(t),\qquad v(t)=(1,t,t^2,\ldots,t^n).
\]
Since $M(t)\geq0$ for all $t\in\mathbb R$, we have $M(t)\in M_{n+1}(\mathcal P_{2n})^+$. Applying the complete positivity of $\psi_F$, we obtain
\[
\psi_F^{(n+1)}(M)=\big(N_{i+j}\big)_{i,j=0}^{n}\geq 0.
\]
But this block Hankel matrix is exactly $\phi^{-1}(F)$ at the $m$-th matrix level. Hence $\phi^{-1}$ is completely positive.

It remains to check the units. The complete order isomorphism sends the chosen strictly positive Hankel unit $\eta$ to
\[
\phi(\eta)(p)=L_\eta(p)=\sum_{k=0}^{2n}a_k\eta_k,
\qquad p(t)=\sum_{k=0}^{2n}a_kt^k.
\]
This is a faithful positive functional and an Archimedean matrix order unit on $\mathcal P_{2n}^d$, so $\phi$ is unital for the convention above. For the choice in Example~\ref{11}, $L_\eta$ is integration against the standard Gaussian measure.

By Theorem~\ref{4}, the canonical embedding $J:\mathcal P_{2n}\longrightarrow\mathcal P_{2n}^{dd}$ is a complete order isomorphism. The dual map $\phi^d\circ J$ therefore identifies $\mathcal P_{2n}$ with $\mathcal H_n^d$, and
\[
(\phi^d\circ J)(e_n)(H)=L_H(e_n)=\sum_{j=0}^n h_{2j}=\Tr(H).
\]
Thus this identification is unital when $\mathcal H_n^d$ is equipped with $\Tr|_{\mathcal H_n}$. A faithful positive functional need not be normalized to a state to serve as a dual order unit. \qedhere

\end{proof}

\begin{cor}
    The operator system $\mathcal H_n$ is the maximal operator system structure on the $(n+1)\times (n+1)$ Hankel matrices.
\end{cor}

\begin{proof}
Since $$\mathcal P_{2n}=\OMIN(\mathcal P_{2n}),$$
finite-dimensional OMIN/OMAX duality implies that $\mathcal P_{2n}^d$ is maximal. Theorem \ref{Iso} then yields that $\mathcal H_n$ is maximal.

\end{proof}

\begin{remark}\label{20}
The polynomial unit is determined by the trace pairing, rather than by the choice of Hankel unit $\eta$. Indeed,
\[
L_H(e_n)=\sum_{j=0}^n h_{2j}=\Tr(H).
\]
The projection construction below explains why this functional, and hence $e_n$, is natural in the matrix realization. Section~\ref{app:trace-unit} identifies the resulting projected operator system explicitly with $\mathcal P_{2n}$ and compares its unit with the one transported by Cayley.
\end{remark}

\begin{remark}
Let $V\subseteq M_N(\mathbb C)$ be a complex $*$-vector subspace that is invariant under transpose and contains a positive definite matrix $E$. Let $P_V:M_N(\mathbb C)\longrightarrow V$ be the orthogonal projection for the Hilbert--Schmidt inner product $\langle A,B\rangle=\Tr(B^*A)$. There are two matrix orders to distinguish:
\[
C_k=M_k(V)\cap M_k(M_N)^+,\qquad
\widetilde C_k=P_V^{(k)}\big(M_k(M_N)^+\big),
\qquad e_V=P_V(I_N).
\]
The inherited cones $C_k$ have order unit $E$. The projected cones $\widetilde C_k$ have order unit $e_V$, as proved below. No positivity of $P_V$ into the inherited cones is assumed, and $e_V$ need not be positive in that inherited order. The transpose-invariance hypothesis ensures that the trace-transpose pairing identifies the dual with the same vector space $V$; it holds in particular for the Hankel spaces.
\end{remark}

The following theorem gives a projection formulation of the quotient description of operator-system duals. See~\cite{45}.

\begin{thm}\label{30}
With the preceding notation, $(V,\{\widetilde C_k\},e_V)$ is an operator system. The map
\[
\Phi:(V,\{\widetilde C_k\},e_V)\longrightarrow
(V,\{C_k\},E)^d,\qquad
\Phi(v)(w)=\Tr(v^T w),
\]
is a unital complete order isomorphism when the dual is equipped with the order unit $\Tr|_V$.
\end{thm}

\begin{proof}
Choose $\delta>0$ with $E\geq\delta I_N$. If $A_j\in M_k(M_N)^+$ and $P_V^{(k)}(A_j)$ converges, then
\[
\delta\Tr(A_j)\leq\Tr\big((I_k\otimes E)A_j\big)
=\Tr\big((I_k\otimes E)P_V^{(k)}(A_j)\big).
\]
The right-hand side is bounded, so the positive matrices $A_j$ are bounded and have a convergent subsequence. Its limit is a positive lift of the limiting projected element. Hence each $\widetilde C_k$ is closed. The same identity gives
\[
\ker P_V^{(k)}\cap M_k(M_N)^+=\{0\}.
\]
If $X$ and $-X$ have positive lifts $A$ and $B$, then $A+B$ belongs to this intersection, so $X=0$. Thus the projected cones are proper. Compatibility with direct sums and scalar compressions follows from the corresponding properties of matrix positivity and the entrywise action of $P_V$.

For $X\in M_k(V)_{sa}$, take $r$ sufficiently large that $rI_{kN}\pm X\geq0$. Applying $P_V^{(k)}$ gives
\[
r(I_k\otimes e_V)\pm X\in\widetilde C_k.
\]
Closedness gives the Archimedean property, proving that the projected structure is an operator system.

The hypotheses on $V$ imply that the bilinear pairing $\Tr(v^T w)$ is nondegenerate and that
\[
\Tr\big((P_V A)^T w\big)=\Tr(A^T w)\qquad(w\in V).
\]
In particular, $\Phi$ is a linear bijection. If $X=[x_{ij}]\in\widetilde C_k$ has a positive lift $A=[A_{ij}]\in M_k(M_N)^+$, then the map associated with $\Phi^{(k)}(X)$ is the restriction to $V$ of
\[
w\longmapsto\big[\Tr(A_{ij}^T w)\big]_{i,j=1}^k.
\]
Its Choi matrix is $A$ with its tensor factors interchanged, hence is positive. This proves that $\Phi$ is completely positive.

Conversely, suppose $\Phi^{(k)}(X)$ is positive, so the associated map $V\longrightarrow M_k$ is completely positive. The space $E^{-1/2}VE^{-1/2}$ is a unital operator subsystem of $M_N$. Conjugating to this subsystem, applying Arveson's extension theorem, and conjugating back gives a completely positive extension to $M_N$. Its positive Choi matrix, after interchanging tensor factors, gives a positive matrix $A=[A_{ij}]$ representing the extension by the preceding formula. Nondegeneracy of the pairing implies $P_V^{(k)}(A)=X$. Thus $\Phi^{-1}$ is completely positive.

Finally,
\[
\Phi(e_V)(w)=\Tr\big((P_V I_N)^T w\big)=\Tr(w).
\]
The restriction of the trace is faithful on the inherited positive cone, so it is a dual order unit and $\Phi$ is unital. \qedhere
\end{proof}

\begin{remark}\label{46}
For $V=\mathcal H_n$, transpose acts as the identity and the pairing in Theorem~\ref{30} is simply $\Tr(KH)$. The inherited Hankel system and the projected Hankel system generally have different matrix cones. Section~\ref{app:trace-unit} shows that the latter has the concrete polynomial realization
\[
K\longmapsto v(t)^*Kv(t),
\]
which sends $P_{\mathcal H_n}(I_{n+1})$ to $e_n$. The coefficient pairing of Theorem~\ref{Iso} is exactly this trace pairing in polynomial coordinates.
\end{remark}

We record the following lemmas, which are fairly standard, for convenience.

\begin{lem}\label{17}
Let $S_1$ and $S_2$ be operator systems. If $C_e^*(S_2)$ is abelian, then every positive linear map $\phi:S_1\to S_2$ is completely positive.
\end{lem}

\begin{proof}
The inclusion of $S_2$ into its $C^*$-envelope is a complete order embedding. Composing $\phi$ with this inclusion gives a positive map from $S_1$ into an abelian $C^*$-algebra, hence a completely positive map. Therefore $\phi$ itself is completely positive. \qedhere
\end{proof}

\begin{lem}\label{43}

Given $S$ a finite-dimensional operator system such that the $C^*$-envelope of its dual space is commutative, then any positive linear map $\varphi: S\longrightarrow B(H)$, for some Hilbert space $H$, is completely positive.

\end{lem}

\begin{proof}
We notice that the image of S under $\varphi$ is a $\ast$-closed subspace of $B(H)$. Then the subspace $\varphi(S)$ and the unit of $B(H)$ span a finite dimensional operator system, denoted by $V$ and $\varphi: S \longrightarrow V$ is a positive linear map. Its dual map, $\varphi^d: V^d\longrightarrow S^d$, is also a positive linear map. By Lemma \ref{17}, $C^*_e(S^d)$ is commutative, then the dual map is completely positive, then $\varphi$ is a completely positive linear map, by $(\varphi^d)^d=\varphi$ and Theorem \ref{5}. \qedhere

\end{proof}

\begin{thm}\label{44}
Every positive linear map $$\varphi: \mathcal{H}_n\longrightarrow T$$
into an operator system T is completely positive.
\end{thm}

\begin{proof}

Since $\varphi:\mathcal H_n\to T$ is positive, $(\mathcal H_n)^d\cong \mathcal P_{2n}$ and the $C^*$-envelope of $\mathcal P_{2n}$ is abelian, Lemma \ref{43} applies and implies that $\varphi$ is completely positive. \qedhere

\end{proof}

\section{The trace order unit and the unital isomorphism class}
\label{app:trace-unit}

We now identify the projected Hankel system of Theorem~\ref{30} with the polynomial system and determine its order unit. This also explains why, for $n\geq2$, the Cayley correspondence does not preserve the unital structure selected by the trace. Let $P_{\mathcal H_n}$ denote the trace-orthogonal projection onto $\mathcal H_n$, write $v(t)=(1,t,\ldots,t^n)^T$, and set
\[
q:M_{n+1}(\mathbb C)\longrightarrow\mathcal P_{2n},\qquad
q(A)(t)=v(t)^*Av(t)=\sum_{i,j=0}^n A_{ij}t^{i+j}.
\]
Writing $m_\ell=\#\{(i,j):0\leq i,j\leq n,\ i+j=\ell\}$, the projection averages entries on each anti-diagonal:
\[
\big(P_{\mathcal H_n}(A)\big)_{ij}
=\frac{1}{m_{i+j}}\sum_{a+b=i+j}A_{ab}.
\]
Consequently, with $R=q|_{\mathcal H_n}$,
\[
q=R\circ P_{\mathcal H_n},\qquad
R(b_\ell)=m_\ell t^\ell\quad(0\leq\ell\leq2n).
\]
In particular, $R$ is a linear bijection and $\ker q=\mathcal H_n^\perp$.

At matrix level $k$, the map $q^{(k)}$ is compression by $I_k\otimes v(t)$, so it takes positive matrices to pointwise positive matrix polynomials. Conversely, if $F\in C'_k$, Theorem~\ref{14} gives $F(t)=Q(t)^*Q(t)$ with $\deg Q\leq n$. Write $Q(t)=D(I_k\otimes v(t))$ for its coefficient matrix $D$. Then $D^*D\geq0$ and
\[
q^{(k)}(D^*D)=F,\qquad
C'_k=q^{(k)}\big(M_k(M_{n+1})^+\big)
=R^{(k)}(\widetilde C_k).
\]
Thus $R$ is a complete order isomorphism from the projected Hankel system to the polynomial system, and
\[
R\big(P_{\mathcal H_n}(I_{n+1})\big)=q(I_{n+1})=e_n.
\]
This proves that the polynomial unit is precisely the projected identity in polynomial coordinates. It also exhibits $q$ as a unital completely positive surjection with positive lifts at every matrix level. Positivity here concerns the projected cones on its Hankel factor, not the inherited matrix cones.

For every Hankel matrix $H$ we have
\begin{equation}\label{app:eq-trace-quotient}
L_H(q(A))=\Tr(AH),\qquad q(I_{n+1})=e_n,
\qquad L_H(e_n)=\Tr(H).
\end{equation}
The first equality uses $H^T=H$ and the ordinary, unnormalized matrix trace. Since $H$ is orthogonal to the kernel of the projection,
\[
\Phi\big(P_{\mathcal H_n}(A)\big)(H)
=\Tr\big(P_{\mathcal H_n}(A)H\big)
=\Tr(AH)=L_H(q(A)).
\]
Hence the projection duality of Theorem~\ref{30} agrees with the coefficient duality of Theorem~\ref{Iso}. The order unit on $\mathcal H_n^d$ corresponding to $e_n$ is exactly $\Tr|_{\mathcal H_n}$, independently of the strictly positive unit chosen on $\mathcal H_n$ itself.

To compare this unit with the Cayley unit, let $D_n=\operatorname{diag}(\binom n0,\ldots,\binom nn)$. Then
\[
q(D_n)=g_n,\qquad L_H(g_n)=\Tr(D_nH).
\]
Moreover, if $K=BHB^*$, with $B$ as in Proposition~\ref{app:prop-cayley}, then
\begin{equation}\label{app:eq-transported-trace}
\Tr(H)=\Tr\big((BB^*)^{-1}K\big).
\end{equation}
Thus the trace functional becomes a specific faithful weighted trace on the Toeplitz space. The underlying matrix orders are identified by Cayley, but the unit $e_n=q(I_{n+1})$ is fixed by the projection and trace construction. The next lemma and proposition show that, for $n\geq2$, no alternative order isomorphism removes the difference between $e_n$ and $g_n$.

\begin{lem}\label{app:lem-cone-automorphisms}
Let $\mathcal B_{2n}$ be the real vector space of homogeneous binary
forms of degree $2n$, with cone consisting of forms nonnegative on
$\mathbb R^2$. Every real linear automorphism $A$ of this cone has
the form
\begin{equation}\label{app:eq-binary-automorphism}
(AF)(x,y)=\lambda F(ax+by,cx+dy),
\qquad
\lambda>0,\quad
\begin{pmatrix}a&b\\c&d\end{pmatrix}\in\GL_2(\mathbb R).
\end{equation}
In particular, it preserves the multiplicities of the projective
complex zeros of a nonzero form.
\end{lem}

\begin{proof}
Homogenization identifies the polynomial cone with this cone by
\[
p\longmapsto\widehat p(x,y)=y^{2n}p(x/y)
=\sum_{k=0}^{2n}a_kx^ky^{2n-k}.
\]
Nonnegativity at $y=0$ follows by continuity from nonnegativity at
$y\ne0$. The extreme rays of the dual cone are exactly evaluation
at real projective points. Indeed, under the coefficient pairing
these are the rank-one Hankel rays, including $E_{nn}$, and
Proposition~\ref{app:prop-vandermonde} below proves that they generate
the full dual cone. Each such ray is extreme because it is rank one
in the positive semidefinite matrix cone. The proof of that
proposition is independent of this lemma.

Projectivizing these rays gives the real rational normal curve
\[
[x:y]\longmapsto[y^{2n}:xy^{2n-1}:\cdots:x^{2n}].
\]
The dual of $A$ induces a projective linear transformation preserving
this curve. Its complexification preserves the complex rational
normal curve, since the real points are Zariski dense. The induced
automorphism of its parameter space $\mathbb P^1_{\mathbb C}$ is
fractional linear: a rational map with a rational inverse has degree
one, by multiplicativity of degree under composition. It is defined
over $\mathbb R$, so it is represented by a matrix in
$\GL_2(\mathbb R)$.

The substitution action of that matrix on forms induces the same
projective action on the dual curve. Two projective linear maps
agreeing there agree everywhere: any $2n+2$ distinct points of the
curve form a projective frame, by the Vandermonde determinant.
Thus the actions on forms differ by a nonzero real scalar.
That scalar is positive because both actions preserve the positive
cone. This proves \eqref{app:eq-binary-automorphism}. An invertible
linear change of homogeneous variables preserves root multiplicities.
\end{proof}

\begin{prop}\label{app:prop-distinct-units}
For every $n\geq2$, the operator systems
\[
(\mathcal P_{2n},\{C'_k\},e_n)
\quad\hbox{and}\quad
(\mathcal P_{2n},\{C'_k\},g_n)
\]
are not unitally order isomorphic, and hence are not unitally
completely order isomorphic. In particular, the trace-unit
polynomial system is not unitally completely order isomorphic to
$(\mathcal F_n,1)$. For $n=1$ the two polynomial units coincide.
\end{prop}

\begin{proof}
A unital order isomorphism restricts to a real cone automorphism
on the self-adjoint part. By Lemma~\ref{app:lem-cone-automorphisms},
it would carry the homogenized units to one another by a scalar
and an invertible linear change of variables. However,
\[
\widehat e_n(x,y)=\sum_{j=0}^nx^{2j}y^{2n-2j},
\qquad \widehat g_n(x,y)=(x^2+y^2)^n.
\]
The first form has $2n$ distinct projective complex zeros. In the
chart $y=1$ these are the $(2n+2)$-nd roots of unity other than
$1$ and $-1$, since
\[
(t^2-1)e_n(t)=t^{2n+2}-1;
\]
there is no zero at infinity. The second form has exactly two
distinct zeros, each of multiplicity $n$. For $n\geq2$ this
contradicts preservation of root multiplicities. The assertion
about $\mathcal F_n$ follows from Proposition~\ref{app:prop-cayley}.
\end{proof}

\begin{remark}
The matrix cones on the two polynomial spaces are identical. Their
inequivalence concerns the distinguished order unit and the resulting
unital isomorphism class. The Cayley correspondence therefore
identifies the underlying matrix orders while the trace pairing
selects a different unital polynomial operator system. This
distinction is independent of the choice of Hankel unit $\eta$.
\end{remark}

\section{Block Hankel Matrices} 

Now, we will explore the block Hankel matrices using tensor products. The block Hankel matrix is not only an interesting topic in both pure and applied mathematical research but also a well-developed branch. It is studied either independently or in conjunction with other problems, for example, the Hamburger moment problem. See \cite{25,24,22,23,21}. 

Given any operator system $\mathcal S$, we can easily see that the tensor product $\mathcal S \otimes \mathcal H_n$ is a $*$-complex vector space. Elements of the tensor product are called \emph{block Hankel matrices}.

\begin{thm}\label{24}

Given a $C^*$-algebra $\mathcal A$, let
\[
H=(h_{i+j})_{i,j=0}^{n}\in \mathcal H_n\otimes_{min} \mathcal A
\]
be a Hermitian block Hankel matrix, where $h_k\in \mathcal A$ for each $k$. Then $H$ is positive if and only if there exists a completely positive map
\[
\phi: \mathcal P_{2n} \longrightarrow \mathcal A
\]
such that $\phi(t^k)=h_k$, for each $0\leq k\leq 2n$. 

\end{thm}

\begin{proof}

Any element $H\in \mathcal H_n \otimes \mathcal{A}$ can be expressed as
\[
H=\sum_{k=0}^{2n} b_k\otimes h_k,
\]
where $\{b_k\}$ is the basis of $\mathcal H_n$. By the finite-dimensional duality between positive elements of $S^d\otimes_{min} \mathcal A$ and completely positive maps $S\to \mathcal A$, and by Theorem \ref{Iso}, the positivity of $H$ is equivalent to the complete positivity of the map
\[
\widetilde{\phi_H}:\mathcal P_{2n}\longrightarrow \mathcal A,
\qquad
\widetilde{\phi_H}\left(\sum_{k=0}^{2n}a_kt^k\right)=\sum_{k=0}^{2n}a_kh_k.
\]
In particular, taking $p(t)=t^k$, we get $\widetilde{\phi_H}(t^k)=h_k$. The converse is the same argument in reverse. \qedhere

\end{proof}

\begin{cor}\label{22}
Given any block Hankel matrix
\[
H=(h_{i+j})_{i,j=0}^{n} \in \mathcal H_n\otimes M_{m}(\C),
\]
where each $h_k\in M_m(\C)$, the positivity of $H$ is equivalent to the following condition: for every $r\in\mathbb N$ and every positive matrix polynomial
\[
P(t)=\sum_{k=0}^{2n}a_kt^k\in M_r(\mathcal P_{2n})^+,
\]
one has
\[
\sum_{k=0}^{2n} a_k\otimes h_k\geq0
\]
in $M_r\otimes M_m$.

\end{cor}

\begin{proof}

This is a direct application of the proof of Theorem \ref{Iso}. The block Hankel matrix $H=(h_{i+j})$ is positive if and only if the associated map
\[
\mathcal P_{2n}\longrightarrow M_m,\qquad t^k\longmapsto h_k,
\]
is completely positive. Testing complete positivity on an arbitrary positive element $P(t)=\sum a_kt^k\in M_r(\mathcal P_{2n})^+$ gives exactly the condition
\[
\sum_{k=0}^{2n}a_k\otimes h_k\geq0. \qedhere
\]

\end{proof} 

\begin{thm}\label{25}

Given a $C^*$-algebra $\mathcal A$, let $H=(h_{i+j}) \in \mathcal H_n \otimes_{min} \mathcal A$ be a positive block Hankel matrix, then there exists a completely positive linear map:
$$\phi: \mathcal H_n \otimes_{min} \mathcal P_{2n}\longrightarrow \mathcal H_n \otimes_{min} \mathcal A,$$
such that $\phi(M)=H$, where $M=\sum_{k=0}^{2n}b_k\otimes t^k=(t^{i+j})_{i,j=0}^{n}$.

\end{thm}

\begin{proof}

We can express any positive $H=(h_{i+j}) \in \mathcal H_n \otimes_{min} \mathcal A$ as $H=\sum_{i=0}^{2n} b_i\otimes h_i$. Set $\phi= I_{\mathcal H{_n}}\otimes \widetilde{\phi_H}$, where $I_{\mathcal H{_n}}$ is the identity map on $\mathcal H_n$. Then $\phi$ is a completely positive map and we see that
$$\phi(M)= I_{\mathcal H{_n}}\otimes \widetilde{\phi_H}(\sum_{i=0}^{2n}b_i\otimes t^i)=\sum_{i=0}^{2n}b_i\otimes h_i=H. \qedhere$$ 

\end{proof}

\begin{remark}\label{27}

We can also consider the matrix M in Theorem \ref{25} as a function $M(t): \R \longrightarrow M_{n+1}(\C)$. $M(t)= v(t)v(t)^*$, where $v(t)= (1, t, t^2, \cdots, t^n )^T$. This function maps the real line into a rank-one positive semi-definite $(n+1)\times (n+1)$ matrix. Due to Theorem \ref{24}, we can call this Hankel matrix valued function $M(t)$  the universal positive Hankel matrix.
    
\end{remark}

\section{Vandermonde decomposition and the trace pairing}
\label{app:vandermonde}

The trace pairing also determines the normalization of polynomial states. We now use an augmented Vandermonde decomposition to describe the trace-one section of the positive Hankel cone, and hence the state space of $(\mathcal P_{2n},e_n)$. The decomposition concerns the same matrix order that appears in the Cayley correspondence, while its normalization uses the trace-induced unit established in Section~\ref{app:trace-unit}.

Write $X=\mathbb R\cup\{\infty\}\cong\mathbb{RP}^1$ and define
\begin{equation}\label{app:eq-normalized-rank-one}
R_n(t)=\frac{v(t)v(t)^*}{e_n(t)}\quad(t\in\mathbb R),
\qquad R_n(\infty)=E_{nn}.
\end{equation}
Since $e_n(t)=v(t)^*v(t)$, these matrices are rank-one orthogonal
projections with trace one. The map $R_n:X\to\mathcal H_n$ is
continuous, including at infinity, as is seen by dividing numerator
and denominator by $t^{2n}$.

\begin{lem}
\label{app:prop-vandermonde}
Every positive semidefinite $H\in\mathcal H_n$ has a decomposition
\begin{equation}\label{app:eq-augmented-vandermonde}
H=\sum_{j=1}^r w_jv(t_j)v(t_j)^*+cE_{nn},
\qquad w_j>0,\quad c\geq0,
\end{equation}
with distinct real $t_j$, such that
\[
r+\mathbf1_{\{c>0\}}=\rk H.
\]
If $H>0$, the decomposition can be chosen with $c=0$ and $r=n+1$.
The zero matrix corresponds to an empty sum and $c=0$.
\end{lem}

\begin{proof}
First suppose $H>0$. Choose any real number $h_{2n+1}$, factor
$H=LL^*$, and set
\[
\widetilde H=(h_{i+j+1})_{i,j=0}^n,\qquad
J=L^{-1}\widetilde H (L^{-1})^*,\qquad w=L^*\varepsilon_0,
\]
where $\varepsilon_k$ is the $k$-th standard coordinate vector.
The matrix $J$ is self-adjoint. For $0\leq k<n$,
$\widetilde H\varepsilon_k=H\varepsilon_{k+1}$, so
\[
J^kw=L^*\varepsilon_k\quad(0\leq k\leq n),\qquad
w^*J^{i+j}w=h_{i+j}\quad(0\leq i,j\leq n).
\]
The vector $w$ is cyclic for $J$, because $L^*$ is invertible.
Thus $J$ has $n+1$ distinct real eigenvalues $t_j$, and each
spectral projection $P_j$ has $w_j=w^*P_jw>0$. The spectral theorem
gives $h_k=\sum_jw_jt_j^k$ for $0\leq k\leq2n$, proving the claim
in the positive definite case.

For general $H\geq0$, fix a strictly positive Hankel matrix $\eta$
and apply the preceding construction to $H+\epsilon\eta$ as
$\epsilon\downarrow0$. Normalize each term by
\eqref{app:eq-normalized-rank-one}. Thus
\[
H+\epsilon\eta=\sum_{j=1}^{n+1}
\alpha_{j,\epsilon}R_n(t_{j,\epsilon}),\qquad
\sum_{j=1}^{n+1}\alpha_{j,\epsilon}=\Tr(H+\epsilon\eta).
\]
The weights are bounded and $X$ is compact. Passing to a subsequence
for all $n+1$ terms yields a representation of $H$ by at most $n+1$
nonnegative multiples of the $R_n(t)$. Delete zero weights and
combine repeated points. This is
\eqref{app:eq-augmented-vandermonde} after undoing the normalization.

Any collection of at most $n+1$ distinct projective Vandermonde
vectors is linearly independent. For finite nodes this is the
Vandermonde determinant. If infinity is present, use the first $r$
coordinates to establish independence of the $r$ finite-node columns,
and then the last coordinate for $\varepsilon_n$; here $r\leq n$.
The rank of a positive sum of their outer products is therefore
the number of nonzero terms, proving the rank assertion.
\end{proof}

\begin{remark}[Relation to the classical moment problem]
The extra summand is necessary: $E_{nn}$ has $h_0=0$ and $h_{2n}=1$,
so no positive measure on $\mathbb R$ can represent all its moments.
Under the polynomial pairing it is the functional
$p\mapsto[t^{2n}]p$. The positive definite construction above is
the finite Jacobi-matrix construction of Gaussian quadrature.
The choice of $h_{2n+1}$ parametrizes the representations with
exactly $n+1$ finite atoms, not all representing measures.
The augmented decomposition is classical; See~\cite{46}.
The proof here is included to make the rank bound and the
normalization at infinity explicit.
\end{remark}

\begin{prop}\label{app:cor-trace-base}
Under the pairing $H\mapsto L_H$, the state space of
$(\mathcal P_{2n},e_n)$ is affinely isomorphic to
\[
\{H\in\mathcal H_n:H\geq0,\ \Tr(H)=1\}
=\conv\{R_n(t):t\in X\}.
\]
Its extreme points are precisely the states
\begin{equation}\label{app:eq-evaluation-states}
\rho_t(p)=\frac{p(t)}{e_n(t)}\quad(t\in\mathbb R),
\qquad \rho_\infty(p)=[t^{2n}]p.
\end{equation}
Every state is a convex combination of at most $n+1$ such states,
and can be expressed using exactly $\rk H$ of them when represented
by the trace-one Hankel matrix $H$.
\end{prop}

\begin{proof}
Theorem~\ref{Iso} identifies positive functionals with positive
Hankel matrices, and \eqref{app:eq-trace-quotient} identifies the
state condition with trace one. In
\eqref{app:eq-augmented-vandermonde} the convex weights are
\[
\alpha_j=w_je_n(t_j),\qquad \alpha_\infty=c,
\qquad \sum_j\alpha_j+\alpha_\infty=1.
\]
Each $R_n(t)$ is an extreme point of the full matrix state space,
being a rank-one projection, and hence of this trace-one section.
Conversely, the finite decomposition implies that any extreme
point of the section is one of these projections. The corresponding
functionals are exactly \eqref{app:eq-evaluation-states}.
\end{proof}

Finally, the normalized polynomial realization
\[
\iota:\mathcal P_{2n}\longrightarrow C(X),\qquad
\iota(p)(t)=\rho_t(p),
\]
is a unital complete order embedding. This follows from pointwise
matrix positivity on $\mathbb R$, division by $e_n(t)>0$, and
continuity at infinity. The trace pairing is visible in the identity
\[
\iota(q(A))(t)=\Tr\big(A R_n(t)\big),\qquad t\in X.
\]
Thus normalized point evaluations on the polynomial system are
restrictions of vector states represented by the rank-one
Hankel projections. Vandermonde decomposition expresses every
polynomial state as a convex combination of these restrictions.
The order unit $e_n=q(I_{n+1})$ supplies exactly the normalization
that makes the polynomial and matrix state pictures agree.

\section*{Acknowledgments}

The authors were partially supported by NSF grant DMS-2055155.

\section*{Methods}

OpenAI's ChatGPT 5.6 Sol and ChatGPT 6 Astra were used at the drafting stage to assist in checking arguments and for copyediting. The work on the Hankel--Toeplitz correspondence, the projection and trace-pairing construction, and Vandermonde decomposition is assisted by Opus 5.5 and ChatGPT 6 Astra. All other statements and proofs are the sole work of the authors. The authors take responsibility for the mathematical claims, proofs, references, and final presentation.

\bibliographystyle{amsra}
\bibliography{mybib}


\end{document}